\documentclass[12pt,draftcls,onecolumn]{IEEEtran}

\usepackage{amssymb,latexsym}
\usepackage{graphicx}

\newcommand{\argmax}{\mathop{\mbox{\rm arg\,max}}}

\newtheorem{thm}{Theorem}[section]

\newtheorem{cor}{Corollary}[section]

\begin{document}
\sloppy

\title{An Asymptotically Optimal Strategy for Constrained Multi-armed Bandit Problems}
\author{Hyeong Soo Chang
\thanks{H.S. Chang is with the Department of Computer Science and Engineering at Sogang University, Seoul 121-742, Korea. (e-mail:hschang@sogang.ac.kr).}%
}

\maketitle
\begin{abstract}
For the stochastic multi-armed bandit (MAB) problem from a constrained model 
that generalizes the classical one, we show that 
an asymptotic optimality is achievable by a simple strategy extended from the $\epsilon_t$-greedy strategy.
We provide a finite-time lower bound on the probability of correct selection of an optimal near-feasible arm
that holds for all time steps. Under some conditions, the bound approaches one as time $t$ goes to infinity.
A particular example sequence of $\{\epsilon_t\}$ having the asymptotic convergence rate in the order of $(1-\frac{1}{t})^4$ that holds from a sufficiently large $t$ is also discussed.
\end{abstract}

\section{Introduction}

Consider a stochastic multi-armed bandit (MAB) problem (see, e.g.,~\cite{berry}~\cite{gittins}~\cite{bianchi}, etc.~for in depth cover of the topic and the related ones) where there is a finite set $A$ of arms and one arm in $A$ needs to be sequentially played. Unlike the classical set up, in our case,
when $a$ in $A$ is played at discrete time $t\geq 1$,
the player not only obtains a sample bounded reward $X_{a,t}\in \Re$ 
drawn from an unknown reward-distribution associated with $a$, whose unknown expectation and variance are $\mu_a$ and $\sigma_{R,a}^2$, respectively, 
but also obtains a sample bounded cost $Y_{a,t} \in \Re$ drawn from an unknown cost-distribution associated with $a$, whose unknown expectation and variance are $C_a$ and $\sigma_{C,a}^2$, respectively. 
Sample rewards and costs across arms are all independent for all time steps. That is, 
$X_{a,t}, X_{b,s}, Y_{p,t'}$, and $Y_{q,s'}$ are independent for all $a,b,p,q\in A$ and all $t,s,t',s'\geq 1$.
For any fixed $a$ in $A$, $X_{a,t}$'s and $Y_{a,t}$'s for $t\geq 1$ are identically distributed, respectively.

One of the classical problem's goals is concerned with ``finding" an optimal arm in $\argmax_{a\in A}\mu_a$ by a proper exploration and exploitation process but here is with an optimal \emph{feasible} arm in $\argmax_{a\in A_f}\mu_a$ where 
$A_f := \{ a\in A | C_a \leq C \}$ for some real constant $C$ ($C$ is a problem parameter and we assume that $A_f \neq \emptyset$).
This cost-measure plays a role of \emph{constraint} for the optimality by the reward-measure.
We need to somehow blend a process of estimating the feasibility of each arm into an exploration-exploitation process for estimating the optimality of each arm. 
Because still only one arm needs to be sequentially played at each time, a novel paradigm of algorithm design seem necessary.
Surprisingly, to the author's best knowledge, no work seems exist yet regarding the problem setup here that we call ``constrained MAB" (CMAB). (Note that for the sake of simplicity, we consider one constraint case. It is straightforward to extend our results into multiple-constraints case.)
A closest work would be the paper by Denardo \emph{et al.}~\cite{denardo}~where the values of the problem parameters in their model are all \emph{known}. Linear programming is considered for solving a (constrained) Markov decision process to find an optimal arm (policy). 
On the other hand,
our methodology is associated with a (simulation) process of iteratively updating estimates of the unknown parameter values, e.g., expectations, from samples of reward and cost and playing the bandit with an arm selected based on those information and obtaining new samples for further estimation
eventually finding an optimal arm.

We define a \emph{strategy} (or algorithm) $\pi:=\{\pi_t, t=1,2,...\}$ as a sequence of mappings such that $\pi_t$ maps from the set of past plays and rewards and costs, $H_{t-1}:=(A \times \Re \times \Re)^{t-1}$ if $t\geq 2$ and $\emptyset$ if $t=1$, to the set of all possible distributions over $A$. 
We denote the set of all possible strategies as $\Pi$. 
We let a random variable $I^{\pi}_t$ denote the arm selected by $\pi$ at time $t$.

The notion of the \emph{asymptotic optimality} of a strategy was introduced by Robbins~\cite{robbins} for the classical MAB problem, i.e., when $A_f=A$. 
We re-define it for the CMAB case:
Let $\mu^*=\max_{a\in A_f}\mu_a$ and $A^*_f := \{a\in A_f| \mu_a = \mu^*\}$.
For a given $\pi \in \Pi$, $\pi$ is an \emph{asymptotically optimal} strategy 
if $\sum_{a\in A^*_f} \Pr\{I^{\pi}_t = a\} \rightarrow 1$ as $t \rightarrow \infty$.
Robbins studied a two-arm problem with Bernoulli reward distributions and Bather~\cite{bather} extended the problem into the general case where $|A|\geq 2$ and
provided an asymptotically optimal ``index-based" strategy. At each time each arm's 
certain performance index is obtained and an arm is selected based on the indices.
The key idea of Bather was to ensure that each arm is played infinitely often by
introducing some randomness into the index computation and to make the effect for an arm vanish 
as the number of times the arm has been played increases.
The $\epsilon_t$-greedy strategy~\cite{auer} basically follows Bather's idea for general MAB problems:
Set $\{\epsilon_t\}$ such that $\sum_{t=1}^{\infty} \epsilon_t = \infty$ and $\lim_{t\rightarrow \infty} \epsilon_t= 0$. 
 The sequence ensures that each arm is played infinitely often and that the selection by the strategy becomes completely greedy in the limit. 
 In addition, the value of $\epsilon_t$ plays the role of switching probability between greedy selection of the arm estimated as the current best and uniform selection over $A$. 
As $\epsilon_t$ goes to zero, the effect from uniform selection vanishes and the strategy achieves the asymptotic optimality. 
By analyzing its finite-time upper bound on the probability of wrong selection,
Auer \emph{et al.}~\cite[Theorem 3]{auer} showed the convergence rate to zero is in the order of $t^{-1}$.
Arguably, the $\epsilon_t$-greedy policy is known to be the simplest and the most representative strategy among existing (heuristic) playing strategies for MAB problems.
It is notable that in practice the performance of the (tuned) $\epsilon_t$-greedy strategy seem difficult to be beaten by other competitive algorithms (see, e.g.,~\cite{auer}~\cite{vermorel}~\cite{kuleshov} for experimental studies).

This brief communique's goal is to show in theory that under some conditions, a simple extension of the $\epsilon_t$-greedy strategy, called ``constrained $\epsilon_t$-greedy" strategy achieves the asymptotic (near) optimality for CMAB problems.
Our approach is to establish a finite-time lower bound on the probability of selecting an optimal near-feasible arm that holds for all time $t$ where 
the near-feasibility is measured by some deviation parameter.
A particular example sequence of $\{\epsilon_t\}$ having the asymptotic convergence rate in the order of $(1-\frac{1}{t})^4$ that holds asymptotically from a sufficiently large $t$ is also discussed.

We remark that our goal does not cover developing an algorithm that minimizes the ``expected regret,"~\cite{bianchi}~\cite{lai} over a finite horizon. 
The regret is the expected loss relative to the cumulative expected reward of taking an optimal (feasible) arm due to the fact that the algorithm does not always play an optimal arm. In our terms, it can be written as $\mu^*T-\sum_{a\in A} \mu_a ( \sum_{t=1}^T \Pr\{I^{\pi}_t =a \})$ if $T$ is the horizon size.
The regret is thus related with a finite-time behavior of the algorithm and in particular measures a degree of effectiveness in its exploration and exploitation process.
Even if a relatively big body of the bandit literature has considered minimizing the regret as the objective for MAB problems (see, e.g.,~\cite{bianchi} and the references therein),
this communique focuses on the instantaneous behavior of the algorithm in $\Pr\{I^{\pi}_t =a\}, a\in A^*_f$ over infinite horizon.
Developing an algorithm that achieves a low or the optimal regret \emph{within the CMAB context} seems indeed challenging and we leave this as a future research topic.

In some sense, our approach for the objective of finding an optimal arm in $A^*_f$ is similar to that of constrained 
simulation optimization under the topic of constrained `ranking and selection' in the literature. 
However, most importantly, our setup is \emph{within the model of multi-armed bandit}. 
We do not draw multiple samples of reward and cost at a single time step. We do not impose any assumption on the reward and the cost distributions (e.g., normality). Moreover, ``(approximately) optimal sampling plan" or ``optimal simulation-budget allocation" is not computed in advance 
as these or subset of these are common assumption and approaches in the constrained simulation optimization literature (see, e.g.,~\cite{pasupathy}~\cite{hunter}~\cite{park} and the references therein). 
%

\section{Algorithm}
\label{sec:algo}

Once $I^{\pi}_t$ in $A$ is realized by the constrained $\epsilon_t$-greedy strategy (referred to as $\pi$ in what follows) at time $t$, the bandit is played with the arm and a sample reward of $X_{I^{\pi}_t,t}$ and a sample cost of $Y_{I^{\pi}_t,t}$ are obtained independently.
We let $T_a(t) := \sum_{n=1}^t [ I^{\pi}_n = a ]$ denote the number of times $a$ has been selected by $\pi$ during the first $t$ time steps, where $[\cdot]$ denotes the indicator function, i.e., $[ I^{\pi}_t = a ]=1$ if $I^{\pi}_t=a$ and 0 otherwise.
The sample average-reward $\bar{X}_{T_a(t)}$ for $a$ in $A$ is then given such that $\bar{X}_{T_a(t)} = \frac{1}{T_a(t)} \sum_{n=1}^{t} X_{a,n} [ I^{\pi}_n = a ]$ 
if $T_a(t)\geq 1$ and 0 otherwise,
where $X_{a,n}$ is the sample reward observed at time $n$ by playing $a$ as mentioned before. 
Similarly, the sample average-cost $\bar{Y}_{T_a(t)}$ for $a$ in $A$ is given such that $\bar{Y}_{T_a(t)} = \frac{1}{T_a(t)} \sum_{n=1}^{t} Y_{a,n} [ I^{\pi}_n = a ]$ if $T_a(t)\geq 1$ and 0 otherwise, where $Y_{a,n}$ is the sample cost observed at time $n$ by playing $a$. 
Note that $E[X_{a,t}]=\mu_a$ and $E[Y_{a,t}]=C_a$ for all $t$.

We refer to the process of selecting an arbitrary arm $a$ in $A$ with the same probabilities of $1/|A|$ for the arms in $A$ as \emph{uniform selection} $U$ \emph{over} $A$
and the selected arm by the uniform selection over $A$ is denoted as $U(A)$.
We formally describe the constrained $\epsilon_t$-greedy strategy, $\pi$, below.
\\
\noindent\textbf{The constrained $\epsilon_t$-greedy strategy}
\begin{itemize}
\item[1.] \textbf{Initialization:} Select $\epsilon_t \in (0,1]$ for $t=1,2,...$ Set $t=1$ and $T_{a}(0) = 0$ for all $a\in A$ and $\bar{X}_{0}=\bar{Y}_0=0$.
\item[2.] \textbf{Loop:} 
\begin{itemize}
\item[2.1] Obtain $A_t = \{a\in A | T_a(t) \neq 0 \wedge \bar{Y}_{T_a(t)} \leq C \}$.
\item[2.2] With probability $1-\epsilon_t$, \\
    $\mbox{      }$ \textbf{Greedy Selection:} $I_{t}^{\pi} \in \argmax_{a \in A_t} \bar{X}_{T_a(t)}$ if $A_t\neq \emptyset$ (ties broken arbitrarily).\\ \mbox{  }
    \hspace{3.3cm} Otherwise, $I_t^{\pi} =U(A)$.\\
And with probability $\epsilon_t$,\\
    $\mbox{      }$ \textbf{Random Selection:} $I_t^{\pi} = U(A)$.
\item[2.3] Play the bandit with $I_t^{\pi}$ and obtain $X_{I^{\pi}_t,t}$ and $Y_{I^{\pi}_t,t}$ independently.
\item[2.4] $T_{I^{\pi}_t}(t) \leftarrow T_{I^{\pi}_t}(t-1) + 1$ and $t\leftarrow t+1$.
\end{itemize}
\end{itemize}
\vspace{0.5cm}

\section{Convergence}
\label{sec:conv}

To analyze the behavior of the constrained $\epsilon_t$-greedy strategy, we define a set of approximately feasible arms: 
A set $A_f^{\pm \delta}$ in $\mathcal{P}(A)$ is a \emph{$\delta$-feasible set} of arms for $\delta \geq 0$ if $A^{-\delta}_f \subseteq A_f^{\pm \delta} \subseteq A^{\delta}_f$ where $A^{\kappa}_f := \{a\in A| C_a \leq C + \kappa, \kappa \in \Re \}$ and $\mathcal{P}(A)$ is the power set of $A$.
(Note that a $\delta$-feasible set may not be unique for a fixed $\delta$ except when $\delta=0$.)
We say that an arm $a$ in $A$ is \emph{$\delta$-feasible} for a given $\delta\geq 0$ if $a$ is in some $\delta$-feasible set.
In the sequel, we further assume that the reward and the cost distributions all have the support in $[0,1]$ for simplicity. That is, $X_{a,t}$ and $Y_{a,t}$ are in $[0,1]$ for any $a$ and $t$.

The following theorem provides a lower bound on the probability that the arm selected by $\pi$ at $t$ is equal to a best arm in some $\delta$-feasible set $A^{\pm \delta}_f$ in terms of the parameters, $\{\epsilon_t\}$, $|A|$, $\delta$, and $\rho:=\min_{a,b\in A}|\mu_a - \mu_b|$.
\begin{thm}
\label{thm:main}
Let $x_t := \frac{1}{2|A|} \sum_{n=1}^{t} \epsilon_n$ for all $t\geq 1$. Then for all $\delta \geq 0$ and $t\geq 1$, we have that
\begin{eqnarray*}
\lefteqn{\Pr \biggl \{ I^{\pi}_t \in \argmax_{a\in A^{\pm \delta}_f} \mu_a \mbox{ for some } \delta\mbox{-feasible } A^{\pm\delta}_f \in \mathcal{P}(A) \biggr \}}\\
& & \geq  \left (1 - \frac{\epsilon_t}{|A|} \right ) ( 1-|A|e^{-\frac{x_t}{5}} ) (1 - 2|A|e^{-2\delta^2 x_t}) (1-2|A| e^{-\frac{\rho^2}{2} x_t}).
\end{eqnarray*}
\end{thm}
\vspace{0.3cm}
In the proof below, some parts are based on the proof idea of the results in~\cite{auer} and~\cite{wang}. 
Before the proof,
note that for any $t$ when for all $a\in A$, $T_a(t) \geq x_t$, it is not possible that $\sum_{a\in A} x_t > t$. This is because $|A|x_t \leq t/2$ from $0 < x_t\leq \frac{t}{2|A|}$, where this comes from the condition that $\epsilon_t \in (0,1]$ for all $t\geq 1$. 

We can see from the lower bound that the conditions of $\sum_{t=1}^{\infty}\epsilon_{t} = \infty$ and $\epsilon_t\rightarrow 0$ as $t \rightarrow \infty$ are necessary for the convergence to one as $t\rightarrow \infty$ because this makes $x_t\rightarrow \infty$.
Furthermore, the lower bound shows that the convergence speed depends on the values of $\delta$ and $\rho$.
If $\rho\neq 0$ but close to zero, the strategy will need a sufficiently large number of samples (depending on the value of $\delta$) to distinguish the arms with the almost same (by $\rho$) values of the reward expectations. For the case where $\rho=0$, we discuss below.

Let $\eta:=\min_{a\in A} |C_a - C|$.
The value of $\eta$ represents another degree of the problem difficulty. Suppose that $\eta\neq 0$.
Then the convergence to selection of an optimal 0-feasible arm at $t\rightarrow \infty$ is guaranteed with any $\delta$ in $(0, \eta)$ under some conditions (cf., Corollary~\ref{cor1}).
If $\delta$ is close to zero, because $\eta$ is close to zero, $x_t$ needs to be sufficiently large to compensate the small $\delta$.
Because $\Pr\{\forall a\in A \mbox{ } T_a(t) \geq x_t\}$ approaches one as $x_t$ increases (cf., the proof below), this means that a large number of samples for each arm is necessary in order for $\pi$ to figure out the feasibility with a high confidence.
The convergence would be slow in general.
At the extreme case, if $\eta=0$ or if there exists an arm that satisfies the constraint \emph{by equality}, then $\delta$ should be zero for the convergence because the 0-feasible set is uniquely equal to $A_f$. 
In this case or the case where $\rho=0$, the lower bound in the theorem statement does not provide any useful result. (We provide a related remark in the conclusion.) 
The asymptotic optimality needs to be approximated by asymptotic near-optimality by fixing $\delta$ and/or $\rho$ (arbitrarily) close to zero.

Finally, if the value $x_t$ of the (normalized) cumulative sum of the switching probabilities up to time $t$ is small, e.g., if the strategy spends rather more on greedy selection (exploitation) than random selection (exploration), the speed would be slow.
That is, the convergence speed depends on the degree of switching between exploration and exploitation. We now provide the proof of Theorem~\ref{thm:main}.

\begin{proof}
We first observe that
the probability that a $\delta$-feasible current-best arm is selected at time $t$ by $\pi$ from some $\delta$-feasible set for a given $\delta \geq 0$ is lower bounded as follows:
\begin{eqnarray}
\lefteqn{\Pr \biggl \{ I^{\pi}_t \in \argmax_{a\in A^{\pm \delta}_f} \mu_a \mbox{ for some } \delta\mbox{-feasible set } A^{\pm\delta}_f \in \mathcal{P}(A) \biggr \}}\nonumber \\
& & \geq \left (1 - \frac{\epsilon_t}{|A|} \right ) \Pr\{ \forall a\in A \mbox{ } T_{a}(t) \geq x_t \} \label{bound1}\\
& & \times \Pr\{ A^{-\delta}_f \subseteq A_t \subseteq A^{\delta}_f | \forall a\in A \mbox{ } T_{a}(t) \geq x_t \}\label{bound2}\\
& & \times \Pr \bigl \{ I^{\pi}_t \in \argmax_{a\in A_t} \mu_a | A^{-\delta}_f \subseteq A_t \subseteq A^{\delta}_f \wedge \forall a\in A \mbox{ } T_{a}(t) \geq x_t \bigr \} \label{bound3}
\end{eqnarray}
We now provide a lower bound for each probability term except $(1-\epsilon_t/|A|)$ in the product as given above. 

Let $T^{R}_a(t)$ be a random variable whose value is the number of plays in which arm $a$ was chosen at random
by uniform selection (denoted as $U^{R}$) in \textbf{Random Selection} of the step 2.2 up to time $t$. That is, $T^{R}_a(t) = \sum_{n=1}^t [I^{\pi}_n = U^{R}(A)]$.
Then for the first $\Pr$ term in~(\ref{bound1}), we have that
\begin{eqnarray*}
\lefteqn{\Pr\{ \forall a\in A \mbox{ } T_{a}(t) \geq x_t \} \geq \Pr\{ \forall a\in A \mbox{ } T_{a}^{R}(t) \geq x_t \}} \\
& & = 1 - \Pr\{\exists a\in A \mbox{ } T^{R}_{a}(t) < x_t \} \\
& & = 1 - \sum_{a\in A} \Pr\{T^{R}_{a}(t) < x_t \} \mbox{ by Boole's inequality (Union bound)} \\
& & \geq 1 - \sum_{a\in A} \Pr\{T^{R}_{a}(t) \leq x_t \} 
\end{eqnarray*} We then apply Bernstein's inequality~\cite{uspensky} (stated for the completeness):
Let $X_1,...,X_j$ be random variables with range [0,1] and $\sum_{i=1}^{j} \mbox{Var}[X_i| X_{i-1},...,X_1] = \sigma^2$.
Let $S_j=X_1 + \cdots + X_j$. Then for all $h\geq 0$, 
\[ \Pr\{ S_j \leq E[S_j] - h  \} \leq e^{-\frac{h^2/2}{\sigma^2+h/2}}.
\] 
Because $E[T_a^{R}(t)] = \frac{1}{|A|} \sum_{n=1}^t \epsilon_n = 2x_t$ and $\mbox{Var}[T^{R}_a(t)] = \sum_{n=1}^t \frac{\epsilon_n}{|A|} (1 - \frac{\epsilon_n}{|A|}) \leq \frac{1}{|A|} \sum_{n=1}^t \epsilon_n = 2x_t$ by observing that $T_a^{R}(t)$ is the sum of $t$ independent Bernoulli random variables,
we have that by substituting $T_a^{R}(t)$ into $S_t$,
\[
  \Pr\{ T^{R}_a(t) \leq 2x_t - x_t \} \leq e^{-\frac{x_t^2/2}{\sigma^2+x_t/2}} \leq e^{-\frac{x_t^2}{2x_t+x_t/2}} = e^{-\frac{x_t}{5}}.
\] It follows that
\[
    \Pr\{\forall a\in A \mbox{ } T^{R}_a(t) \geq x_t \} \geq 1-\sum_{a\in A} e^{-\frac{x_t}{5}} = 1-|A|e^{-\frac{x_t}{5}}.
\]

For the second probability term in~(\ref{bound2}), letting the event $\{ \forall a\in A \mbox{ } T_{a}(t) \geq x_t \}$ be $E$
\begin{eqnarray*}
\lefteqn{\Pr\{ A^{-\delta}_f \subseteq A_t \subseteq A^{\delta}_f | \forall a\in A \mbox{ } T_{a}(t) \geq x_t \}}\\
& &  = 1- \Pr\{ \exists a\in A \mbox{ } \bar{Y}_{T_{a}(t)} - C_a > \delta | E\} - \Pr\{ \exists a\in A \mbox{ } \bar{Y}_{T_{a}(t)} - C_a < -\delta | E\} \\
& & = 1 - \sum_{a\in A}\Pr\{ \bar{Y}_{T_{a}(t)} - C_a > \delta |E\} - \sum_{a\in A}\Pr\{ \bar{Y}_{T_{a}(t)} - C_a < -\delta |E\} \\
& & \geq 1 - \sum_{a\in A} e^{-2\delta^2T_{a}(t)} - \sum_{a\in A} e^{-2\delta^2T_{a}(t)} \\
& & \geq 1 - 2|A|e^{-2\delta^2 x_t},
\end{eqnarray*} where the lower bound on the last equality is achieved by Hoeffding's inequality~\cite{hoeff}: For random variables $X_1,...,X_j$ with range $[0,1]$ such that $E[X_i|X_1,...,X_{i-1}]=\gamma$ for all $i$, $\Pr\{ X_1 + \cdots + X_j \} \leq j\gamma - h\} \leq e^{-2h^2/j}$ for all $h \geq 0$.

For the third probability term in~(\ref{bound3}), let $i_t^*$ denote any fixed arm in the set $\argmax_{a\in A_t} \mu_a$. 
Let $\Delta_a = \mu_{i^*_t} - \mu_a$ for $a\in  A_t \setminus \{i_t^*\}$. 
Then letting the event $\{A^{-\delta}_f \subseteq A_t \subseteq A^{\delta}_f \wedge \forall a\in A \mbox{ } T_{a}(t) \geq x_t \}$ be $E'$
\begin{eqnarray*}
\lefteqn{\Pr\{ I^{\pi}_t \notin \argmax_{a\in A_t} \mu_a | A^{-\delta}_f \subseteq A_t \subseteq A^{\delta}_f \wedge \forall a\in A \mbox{ } T_{a}(t) \geq x_t \}} \\
& &  \leq  \sum_{a\in A_t \setminus \argmax_{b\in A_t}\mu_b} \biggl ( \prod_{c\in \argmax_{b\in A_t}\mu_b} \Pr\{ \bar{X}_{T_a(t)} > \bar{X}_{T_{c}(t)} | E'\} \biggr ) \\
& &  \leq  \sum_{a\in A_t \setminus \{i_t^*\}} \Pr\{ \bar{X}_{T_a(t)} > \bar{X}_{T_{i^*_t}(t)} | E'\} \\
& &  \leq  \sum_{a\in A_t \setminus \{i_t^*\}} \Pr\{ \bar{X}_{T_a(t)} > \mu_a + \frac{\Delta_a}{2} |E'\} + \Pr\{ \bar{X}_{T_{i^*_t}(t)} <  \mu_{i^*_t} - \frac{\Delta_a}{2} |E'\} \\
& & \leq \sum_{a\in A_t \setminus \{i_t^*\}} \bigl ( e ^{-2 (\frac{\Delta_a}{2})^2 T_a(t)} + e ^{-2 (\frac{\Delta_a}{2})^2 T_{i^*_t}(t)} \bigr ) \mbox{ by Hoeffding's inequality}\\
& & \leq \sum_{a\in A_t \setminus \{i_t^*\}} 2 e ^{-2 (\frac{\Delta_a}{2})^2 x_t} \leq 2|A| e ^{-2 (\frac{\min_{a,b\in A}|\mu_a - \mu_b | }{2})^2 x_t}.
\end{eqnarray*} It follows that the third term is lower bounded by $1-2|A| e ^{-2 (\frac{\rho}{2})^2 x_t}$.

Putting the lower bounds of the three probability terms in~(\ref{bound1}),~(\ref{bound2}), and~(\ref{bound3}) together, we have the stated result that
\begin{eqnarray*}
\lefteqn{\Pr \biggl \{ I^{\pi}_t \in \argmax_{a\in A^{\pm \delta}_f} \mu_a \mbox{ for some } \delta\mbox{-feasible } A^{\pm\delta}_f \in \mathcal{P}(A) \biggr \}}\\
& & \geq  \left (1 - \frac{\epsilon_t}{|A|} \right ) ( 1-|A|e^{-\frac{x_t}{5}} ) (1 - 2|A|e^{-2\delta^2 x_t}) (1-2|A| e^{-\frac{\rho^2}{2} x_t}).
\end{eqnarray*}
\end{proof}
The following corollary is immediate. It states that the asymptotic optimality is achievable by $\pi$ when $\eta \neq 0$ and $\rho\neq 0$ under the conditions on $\{\epsilon_t\}$.
\begin{cor}
Suppose that $\sum_{t=1}^{\infty} \epsilon_t = \infty$ and $\lim_{t\rightarrow \infty} \epsilon_t = 0$ and that $\eta\neq 0$ and $\rho\neq 0$. Then
$\lim_{t\rightarrow \infty} \Pr \{ I^{\pi}_t \in \argmax_{a\in A_f} \mu_a \} = 1$.
\end{cor}
\begin{proof}
From $\sum_{t=1}^{\infty} \epsilon_t=\infty$, $x_t \rightarrow \infty$ as $t \rightarrow \infty$. And $\epsilon_t$ goes to zero and $\rho\neq 0$. Therefore from Theorem~\ref{thm:main}, 
$\lim_{t\rightarrow \infty} \Pr \biggl \{ I^{\pi}_t \in \argmax_{a\in A^{\pm \delta}_f} \mu_a \mbox{ for some } \delta\mbox{-feasible } A^{\pm\delta}_f \in \mathcal{P}(A) \biggr \} = 1$ if $\delta$ is fixed in $(0,\infty)$.
Because $\eta \neq 0$, we observe that $A_f^{-\delta} = A_f = A_f^{\delta}$ for any $\delta\in (0,\eta)$
implying the event $\{ I^{\pi}_t \in \argmax_{a\in A^{\pm \delta}_f} \mu_a \mbox{ for some } \delta\mbox{-feasible } A^{\pm\delta}_f \in \mathcal{P}(A) \}$ is equal to $\{ I^{\pi}_t \in \argmax_{a\in A_f} \mu_a \}$ for such $\delta$.
\end{proof}

We provide a particular example of the sequence $\{\epsilon_t\}$ such that the convergence rate can be obtained.
\begin{cor}
\label{cor1}
Assume that for $t \geq 1$, $\epsilon_t = \min\{1, \frac{k}{t}\}$ where $k > 1$.
Then for $t\geq k$ we have that for any $\delta \geq 0$,
\[
\Pr \biggl \{ I^{\pi}_t \in \argmax_{a\in A^{\pm \delta}_f} \mu_a \mbox{ for some } \delta\mbox{-feasible } A^{\pm\delta}_f \in \mathcal{P}(A) \biggr \} 
\geq \biggl (1 - \frac{k}{|A|t} \biggr ) \biggl ( 1 - \frac{\beta(k,|A|,\delta,\rho)}{t^{\alpha(k,|A|,\delta,\rho)}} \biggr )^3,
\] where $\alpha(k,|A|,\delta,\rho) = \min \{ \frac{k}{10|A|}, \frac{\delta^2 k}{|A|}, \frac{k \rho}{4|A|} \}$
and $\beta(k,|A|,\delta,\rho) = \max \{ |A|k^{\frac{k}{10|A|}}, 2|A|k^{\frac{\delta^2 k}{|A|}}, 2|A|k^{\frac{k \rho}{4|A|}} \}$.
\end{cor}
\begin{proof}
From the assumption on $\{\epsilon_t\}$, $x_t = \frac{1}{2|A|}\sum_{n=1}^{k-1} \epsilon_n + \frac{1}{2|A|}\sum_{n=k}^{t} \epsilon_n = \frac{k-1}{2|A|} + \frac{k}{2|A|}\sum_{n=k}^t \frac{1}{n} \geq  \frac{k-1}{2|A|} + \frac{k}{2|A|} \ln(\frac{t+1}{k}) \geq \frac{k}{2|A|} \ln(\frac{t}{k})$.
Then by using $x_t \geq \frac{k}{2|A|} \ln(\frac{t}{k})$ in the lower bound given in Theorem~\ref{thm:main}, for $t \geq k$ and $\delta\geq 0$,
\begin{eqnarray*}
\lefteqn{\Pr \biggl \{ I^{\pi}_t \in \argmax_{a\in A^{\pm \delta}_f} \mu_a \mbox{ for some } \delta\mbox{-feasible } A^{\pm\delta}_f \in \mathcal{P}(A) \biggr \}} \\
& & \geq \biggl (1 - \frac{k}{|A|t} \biggr ) \biggl ( 1- \frac{|A|k^{\frac{k}{10|A|}}}{t^{\frac{k}{10|A|}}} \biggr ) \biggl ( 1 - \frac{2|A| k^{\frac{\delta^2 k}{|A|}}}{t^{\frac{\delta^2 k}{|A|}}} \biggr ) \biggl ( 1 - \frac{2|A|k^{\frac{k \rho}{4|A|}}}{t^{\frac{k \rho}{4|A|}}} \biggr) \\
& & \geq \biggl (1 - \frac{k}{|A|t} \biggr ) \biggl ( 1 - \frac{\beta(k,|A|,\delta,\rho)}{t^{\alpha(k,|A|,\delta,\rho)}} \biggr )^3.
\end{eqnarray*} 
\end{proof} For example that if $\alpha(k,|A|,\delta,\rho) \geq 1$, i.e., $k \geq \max\{\frac{4|A|}{\rho}, \frac{|A|}{\delta^2}, 10|A| \}$,
$\Pr \biggl \{ I^{\pi}_t \in \argmax_{a\in A^{\pm \delta}_f} \mu_a \mbox{ for some } \delta\mbox{-feasible } A^{\pm\delta}_f \biggr \} = \Theta((1-1/t)^{4})$, i.e., the probability is in the order of $(1-1/t)^{4}$ for $t\geq k$.
In general, if $\delta$ and/or $\rho$ is small, in order to make $\alpha(\cdot)\geq 1$, $k$ needs to be sufficiently large. The convergence rate is achieved asymptotically.

\section{Concluding Remark}

As we mentioned before, if there exists an arm that achieves the equality constraint or if $\eta=0$, then the finite-time bound in Theorem~\ref{thm:main} does not provide any useful result because $\delta$ needs to be set zero. When $\rho=0$, we have the same issue.
It seems that describing a finite-time behavior of the strategy including both cases (e.g., by obtaining a useful finite-time bound) is difficult. We leave this as a future study.
However, we remark that these cases do not break the convergence or the asymptotic optimality of the constrained $\epsilon_t$-greedy strategy.
This is because as long as the condition that $\sum_{t=1}^{\infty} \epsilon_t = \infty$ and $\epsilon_t\rightarrow \infty$ holds, in fact, we still preserve the property that each action in $A$ is played infinitely often in the constrained $\epsilon_t$-greedy strategy. 
This can be seen by the fact that $T_a(t)$ goes to infinity for each $a\in A$ with probability one as $t\rightarrow \infty$. The sample average of $\bar{Y}_{T_a(t)}$ and $\bar{X}_{T_a(t)}$ will then eventually converge to the true average of $C_a$ and $\mu_a$, respectively, in the limit (simply by the law of large numbers).
The probability that the constraint $\epsilon_t$-greedy strategy selects an optimal feasible arm will approach one in the limit.

\end{document}